\documentclass{amsart}
\usepackage{amsfonts}
\usepackage{amsmath}
\usepackage{amssymb}
\newtheorem{theorem}{Theorem}[section]

\newtheorem{lemma}[theorem]{Lemma}

\newtheorem{corollary}[theorem]{Corollary}

\theoremstyle{definition}

\newtheorem{definition}[theorem]{Definition}

\newtheorem{example}[theorem]{Example}

\newtheorem{remark}[theorem]{Remark}

\newtheorem *{Theorem A}{Theorem A}
\newtheorem *{Theorem B}{Theorem B}
\newtheorem *{Corollary B}{Corollary B}
\newtheorem *{Theorem C}{Theorem C}
\newtheorem *{Theorem D}{Theorem D}
\newtheorem *{Projective Schur's Lemma}{Projective Schur's Lemma}

\begin{document}
\title{On the reducibility of exact covering systems}
\author{Ofir Schnabel}
\address{Institute of Algebra and Number Theory, 
University of Stuttgart, Stuttgart 70569, Germany}
\email{os2519@yahoo.com}

\begin{abstract}
There exist irreducible exact covering systems (ECS). These are ECS which are not a proper split
of a coarser ECS. However,
an ECS admiting a maximal modulus which is divisible by at most two distinct primes, primely splits a coarser ECS.
As a consequence, if all moduli of an ECS $A$, are divisible by at most two distinct primes, then
$A$ is natural. That is, $A$ can be formed by iteratively splitting the trivial ECS. 
\end{abstract}
\date{\today\vspace{-0.5cm}}
\maketitle
\bibliographystyle{abbrv}

\section{Introduction}\pagenumbering{arabic} \setcounter{page}{1}
An \textit{exact covering system} (ECS) is a partition of
$\mathbb{Z}$ into finitely many arithmetic progressions
\begin{equation}\label{eq:ecs}
A={\{a_s(n_s)\}_{s=1}^k},
\end{equation}
where $a(n)$ is the arithmetic progression $a+\mathbb{Z}n$. Here, $n$ is
the \textit{modulus} of the arithmetic progression $a(n)$.
An ECS~\eqref{eq:ecs} admits \textit{multiplicity} if there exist
$1\leq i<j\leq k$ such that $n_i=n_j$. The ECS $A=\{0 (1)\}$ is
called the \textit{trivial ECS}.

The concept of ECS was first
introduced by P. Erd\H{o}s in the early 1930's.
A main concern in the research on ECS is finding restraints on the number of times
a modulus occurs in an ECS.
Erd\H{o}s conjectured the following: Every non-trivial ECS
admits multiplicity. Erd\H{o}s conjecture was
proved in the beginning of the 1950's independently by H. Davenport,
L. Mirsky, D. Newman and R. Rado (see \cite{erdos52}). In fact, the proof
shows that such multiplicity occurs in the \textit{greatest}
modulus. This result was generalized by \v{S}. Zn\'{a}m \cite{SZ70} and later by Y.G. Chen and {\v{S}}. Porubsk{\`y} \cite{chen}.
The proofs in these papers use generating functions of an ECS and
a deep relation between the number of times the greatest difference $m$ occurs in an ECS
and minimal vanishing sums of $m$-th roots of unity (see \cite{LL00}).
However, results in the spirit of the above results were obtained in \cite{MR875291},\cite{MR837965} using combinatorical methods.
For a more comprehensive study of ECS the reader is referred to a
monograph by \v{S}. Porubsk\'{y} \cite{porubsky1981} and to a review
by \v{S}. Porubsk\'{y} and J. Sch\"{o}nheim \cite{porubsky1999}.

Our main concern in this note is \textit{reducibility} of ECS. Notice that
for any natural number $n$, there is a \textit{basic} ECS
\begin{equation}
\{i(n)\}_{i=0}^{n-1}.
\end{equation}
This is a \textit{splitting} of the trivial ECS.
In a similar way we can split any ECS by splitting an arithmetic progression
$a(t)$ into $n$ arithmetic progressions
\begin{equation}
\{a+it(tn)\}_{i=0}^{n-1}.
\end{equation}
An ECS is \textit{natural} if it is formed by iteratively splitting the trivial ECS.
\begin{definition}
An ECS $A$ primely splits an ECS $B$, (denote $A\models B$),
if there exists a prime number $p$ such that
\[
B = \left\{ {a_i \left( {n_i } \right)} \right\}_{i = 1}^k ,\quad
A = \left\{ {a_i \left( {n_i } \right)} \right\}_{i = 1}^{k - 1}
\bigcup {\left\{ {a_k  + jn_k (pn_k )} \right\}_{j = 0}^{p - 1} }.
\]
In other words: $A$ is obtained from $B$ by splitting one of the
arithmetic progressions into $p$ arithmetic progressions.
\end{definition}
Throughout this note, maximality will be with respect to the division partial order.
In particular, when given an ECS, a modulus is maximal if it does not divide any other modulus in this ECS.
Our main theorem is the following
\begin{Theorem A}
Suppose that an ECS $A=\{a_s(n_s)\}_{s=1}^k$ has a maximal modulus
of the form $p_1^{k_1}p_2^{k_2}$, where $p_1,p_2$
are primes and $k_1,k_2\geq 0$. Then $A$ primely splits an ECS $B$.
\end{Theorem A}
In a sense, Theorem A discusses reducibility of ECS. Denote the
least common multiple of $B=\{n_1, n_2 ,\ldots , n_k\}\subseteq
\mathbb{N}$ by $N(B)$. For an ECS~\eqref{eq:ecs}, denote the least
common multiple of all the moduli by $N(A)$. Throughout this note
$p_i$ stands for prime number.
 In two papers
\cite{korec1984, korec1986}, I. Korec investigate
ECS with $N(A)=p_1^{k_1}p_2^{k_2}$ and
$N(A)=p_1^{k_1}p_2^{k_2}p_3^{k_3}$. In particular, he discusses the
reducibility (see \cite{korec1984}) of such ECS. In \cite{polach},
I. Pol{\'a}ch generalizes some of Korec's results. The approach
adopted by Korec and Pol{\'a}ch is essentially different from the
classical approach of generating functions. Theorem A is in the
same spirit as Korec and Pol{\'a}ch results. However, our methods
are similar to the classical methods and rely heavily on the
above mentioned relation between ECS and vanishing sums of roots of
unity. As a corollary of Theorem A we get the following result which deals with some natural ECS. In particular, it classifies all
the ECS with $N(A)=p_1^{k_1}p_2^{k_2}$.
\begin{corollary}\label{th:ECS3} 
Let $A={\{a_s(n_s)\}_{s=1}^k}$ be an ECS. If no modulus is divisible by more than two distinct primes then
$$A=A_1\models A_2\models \ldots \models A_{n-1} \models \{0 (1)\}.$$
In particular, this holds in the case when $N(A)$ admits at most two prime factors (see \cite{korec1984}).
\end{corollary}
Another corollary of Theorem A is a restraint on ECS $A$ with $N(A)= p_1^{k_1}p_2^{k_2}p_3^{k_3}$.
\begin{corollary}\label{th:ECS4}
Let $A={\{a_s(n_s)\}_{s=1}^k}$ be an ECS such that $N(A)= p_1^{k_1}p_2^{k_2}p_3^{k_3}$.
Assume also there exist moduli $n_1,n_2,n_3$ such that
\begin{equation}
\begin{array}{ccc}
p_1\mid n_1n_2, & p_2\mid n_1n_3, & p_3\mid n_2n_3, \\
p_1 \nmid n_3, & p_2 \nmid n_2, & p_3 \nmid n_1
\end{array}
\end{equation}
Then $p_1p_2p_3$ divides some modulus $n_j$.
\end{corollary}
{\bf Acknowledgements.}
This paper is a part of the author's M.Sc dissertation under the supervision of Y. Ginosar.
\section{Preliminaries}
Given an ECS~\eqref{eq:ecs} we may always assume that
\begin{equation}\label{eq:cut}
0 \leq a_s < n_s\quad  \text{for all}\quad 1 \leq s \leq k.
\end{equation}
The classical approach for investigating multiplicity is to consider the following generating function.
For $|z|<1$ we have:
\begin{equation}\label{eq:genr}
\sum_{s=1}^k \frac{z^{a_s}}{1-z^{n_s}}= \sum_{s=1}^k
\sum_{q=0}^\infty z^{a_s+qn_s}= \sum_{n=0}^\infty z^n=
\frac{1}{1-z}.
\end{equation}
Let $B=\{n_1,n_2,\ldots ,n_k\}$ be a set of natural numbers.
Assume that $n_r$ is a maximal element in $B$. Denote the least common multiple of $n_1,n_2,\ldots ,n_s$ by $N$.
Let $C_{n_r}$ be the cyclic group of order $n_r$ generated by
$\overline{z}$ and let $\zeta$ be a primitive \textit{$n_r$}-th root
of unity. Consider the following ring homomorphisms:
\begin{align*}
 \gamma : \mathbb{Z}[z] & \rightarrow \mathbb{Z}C_{n_r} \\
  z & \mapsto \overline{z}\\
 \varphi : \mathbb{Z}C_{n_r} & \rightarrow \mathbb{C} \\
  \overline{z} & \mapsto \zeta.
 \end{align*}
Here, $\mathbb{Z}[z]$ is the ring of polynomials with integral coefficients in the variable $z$
and $\mathbb{Z}C_{n_r}$ is the integral group algebra.
The following two lemmas will be used in the proof of Theorem A.
\begin{lemma}\label{th:geo}
With the above notation, let $t$ be a divisor of $N(B)$. Then
$\varphi \circ \gamma\left(\frac{1-z^{N(B)}}{1-z^t}\right)\neq 0$ if and only if $n_r$ divides $t$.
In particular,
\begin{enumerate}
\item
$$\varphi \circ \gamma\left(\frac{1-z^{N(B)}}{1-z}\right)=0.$$
\item By the maximality of $n_r$, for $1\leq i\leq k$
$$\varphi \circ \gamma\left(\frac{1-z^{N(B)}}{1-z^{n_i}}\right)=0$$
if and only if $n_i\neq n_r$.
\end{enumerate}
\end{lemma}
\begin{proof}
First, since $n_r$ divides $N(B)$,
$$\varphi \circ \gamma(1-z^{N(B)})=1-\zeta ^{N(B)}=0.$$
Now, $t$ admits the following decomposition $t=qn_r+r$ such that
$q,r$ are natural numbers and $0\leq r< n_r$. Then, $\varphi \circ
\gamma(1-z^t)=1-\zeta ^t=1-\zeta ^r$. Hence, if $n_r$ is not a
divisor of $t$ we get that
$$\varphi \circ \gamma(\frac{1-z^{N(B)}}{1-z^t})\neq 0.$$
Assume now that $n_r$ is a divisor of $t$ and recall that if
$r_1|r_2$ then
 \begin{equation}\label{eq:poly}
 \frac{1-z^{r_2}}{1-z^{r_1}}=\sum _{i=0}^{\frac{r_2}{r_1}-1} z^{i\cdot r_1}.
 \end{equation}
Then, if we denote $N(B)=cn_r$ and $t=qn_r$ we get that
\begin{equation}
\frac{1-z^N(B)}{1-z^t}=\frac{\sum _{i=0}^{c-1} z^{in_r}}{\sum
_{i=0}^{q-1} z^{in_r}}.
\end{equation}
Hence,
\begin{equation}
\varphi \circ
\gamma\left(\frac{1-z^{N(B)}}{1-z^{t}}\right)=\frac{\overbrace{1+1+\ldots
+1}^{c \text{ times}}}{\underbrace{1+1+\ldots +1}_{q \text{
times}}}\neq 0.
\end{equation}
\end{proof}
Let $P_i=\{\bar{z}^{\frac{jn_r}{p_i}}\}_{j=0}^{p_i-1}$ be the unique subgroup of order $p_i$ in $C_{n_r}$.
Define
$$\sigma _{n_r}(P_i):=\sum _{g\in P_i} g\in \mathbb{Z}C_{n_r}=\sum _{j=0}^{p_i-1}\bar{z}^\frac{j\cdot n_r}{p_i}.$$
\begin{lemma}\label{lemma:ll}(\cite[Theorem 3.3]{LL00})
Let $n_r=p_1^{k_1}p_2^{k_2}$ and let $0\neq x\in \mathbb{N}C_{n_r}\cap$ker$(\varphi)$.
Then $x$ admits one of the following decompositions:
$$x=\overline{z}^d\cdot \sigma _{n_r}(P_1)+ \sum _{i=0}^{n_r-1} b_i\overline{z}^i,\quad
b_i\geq 0,\quad 0\leq d< \frac{n_r}{p_1};$$
or
$$x=\overline{z}^d\cdot \sigma _{n_r}(P_2)+ \sum _{i=0}^{n_r-1} b_i\overline{z}^i,\quad
b_i\geq 0,\quad 0\leq d< \frac{n_r}{p_2}.$$
\end{lemma}
\section{Main part}
\noindent \textit{Proof of Theorem A}.\\
Let $n_r=p_1^{t_1}p_2^{t_2}$
be a maximal modulus in $A$.
Equation~\eqref{eq:genr} can be written in the following way:
\begin{equation}\label{eq:gen}
 \sum_{s=1}^k\frac{z^{a_s}}{1-z^{n_s}}=\sum_{\{s:n_s=n_r\}}\frac{z^{a_s}}{1-z^{n_r}}+
 \sum_{\{s:n_s\neq n_r\}}\frac{z^{a_s}}{1-z^{n_s}}=
 \sum_{n=0}^\infty z^n=\frac{1}{1-z}.
 \end{equation}
 Both sides of~\eqref{eq:gen} are elements in
 $\mathbb{Q}(z)$, the field of rational functions with
rational coefficients in the variable $z$.
As before, denote the least common multiple of $n_1,n_2,\ldots ,n_k$ by $N$.
 By multiplying both sides of~\eqref{eq:gen} by $1-z^N$ we get
\begin{equation}\label{eq:gen1}
 \sum_{s=1}^k\frac{z^{a_s}\cdot (1-z^N)}{1-z^{n_s}}=
 \frac{1-z^N}{1-z}.
 \end{equation}
Hence, both sides of~\eqref{eq:gen1} are in $\mathbb{Z}[z]$. Consequently, by Lemma~\ref{th:geo} the right-hand side of~\eqref{eq:gen1} is in the
kernel of $\varphi \circ \gamma$. Hence the left-hand side is also
in $ker(\varphi \circ \gamma)$. Therefore,
\begin{equation}
\sum_{\{s:n_s=n_r\}}\frac{z^{a_s}\cdot (1-z^N)}{1-z^{n_s}}=
\frac{1-z^N}{1-z^{n_r}}\sum_{\{s:n_s=n_r\}}z^{a_s} \in \text{ker}(\varphi
\circ \gamma).
\end{equation}
Thus, by Lemma~\ref{th:geo} we get:
\begin{equation}\label{eq:uri}
\sum_{\{s:n_s=n_r\}}z^{a_s}\in \text{ker}(\varphi \circ \gamma).
\end{equation}
Hence,
\begin{equation}
\gamma\left(\sum_{\{s:n_s=n_r\}}z^{a_s}\right)=\sum_{\{s:n_s=n_r\}}\overline{z}^{a_s}\in
\text{ker}(\varphi).
\end{equation}
Now, by Lemma~\ref{lemma:ll} we may assume without loss of generality that
\begin{equation}
\gamma\left(\sum_{\{s:n_s=n_r\}}z^{a_s}\right)=\textstyle{\sum_1 +\sum_2},
\end{equation}
where
\begin{equation}
\textstyle{\sum_1}=\overline{z}^d\cdot \sigma
_{n_r}(P_1)=\overline{z}^d\cdot \sum
_{j=0}^{p_1-1}\overline{z}^\frac{j\cdot n_r}{p_1},
\end{equation}
where $d<\frac{n_r}{p_1}$.
And
\begin{equation}
\textstyle{\sum_2}=\sum _{i=0}^{n_r-1} b_i\overline{z}^i,\quad
b_i\geq 0.
\end{equation}
Notice that $\ker \gamma=(z^{n_r}-1)\mathbb{Z}[z]$. Consequently,
the restriction of $\gamma$ to polynomials in $\mathbb{Z}[z]$ with
degree smaller than $n_r$ is 1-1. Let
\begin{equation}
g(z)=z^d\cdot \sum _{j=0}^{p_1-1}z^\frac{j\cdot n_r}{p_1}.
\end{equation}
Then $\gamma (g(z))=\sum_1$. Hence
\begin{equation}
\gamma\left(\sum_{\{s:n_s=n_r\}}z^{a_s}-g(z)\right)=\gamma \left(\sum _{i=0}^{n_r-1}
b_iz^i\right)=\textstyle{\sum_2}.
\end{equation}
By~\eqref{eq:cut}, the degree of
\begin{equation}
\sum_{\{s:n_s=n_r\}}z^{a_s}-g(z)
\end{equation}
is smaller than $n_r$. Therefore by the 1-1 property on such
polynomials,
\begin{equation}
\sum_{\{s:n_s=n_r\}}z^{a_s}-g(z)=\sum _{i=0}^{n_r-1} b_iz^i.
\end{equation}
Consequently,
\begin{equation}\label{eq:pirok}
\sum_{\{s:n_s=n_r\}}z^{a_s}=g(z)+\sum _{i=0}^{n_r-1} b_iz^i=z^d\cdot
\sum _{j=0}^{p_1-1}z^\frac{j\cdot n_r}{p_1}+\sum _{i=0}^{n_r-1}
b_iz^i.
\end{equation}
By~\eqref{eq:gen} and~\eqref{eq:pirok} ,
\begin{equation}\label{eq:union}
\frac{z^d\cdot \sum _{j=0}^{p_1-1}z^\frac{j\cdot n_r}{p_1}}{1-z^{n_r}}+\frac{\sum
_{i=0}^{n_r-1} b_iz^i}{1-z^{n_r}}+\sum_{\{s:n_s\neq n_r\}}\frac{z^{a_s}}{1-z^{n_s}}=\frac{1}{1-z}.
\end{equation}

Notice that
\begin{equation}\label{eq:david}
\frac{z^d\cdot \sum _{j=0}^{p_1-1}z^\frac{j\cdot
n_r}{p_1}}{1-z^{n_r}}=\frac{z^d}{1-z^{\frac{n_r}{p_1}}}.
\end{equation}
So, by~\eqref{eq:union} and by~\eqref{eq:david}
\begin{equation}\label{eq:notlast}
\sum_{\{s:n_s\neq
n_r\}}\frac{z^{a_s}}{1-z^{n_s}}+\frac{z^d}{1-z^{\frac{n_r}{p_1}}}
+\frac{\sum _{i=0}^{n_r-1} b_iz^i}{1-z^{n_r}}=\frac{1}{1-z}.
\end{equation}
Since $b_i\geq 0$, every summand $z^{a_s}$ of the left hand side of
equation~\eqref{eq:pirok} is either a summand of $g(z)$ if
$a_s\equiv d (\text{mod } \frac{n_r}{p_1})$ or a summand in $\sum
_{i=0}^{n_r-1} b_iz^i$. So,~\eqref{eq:notlast} is a generating
function of a new ECS $B$, where $B$ is obtained by consolidation of
the $p_1$ arithmetic progressions:
\begin{equation}
\left\{d(n_r), d+\frac{n_r}{p_1}(n_r), \ldots ,
d+\frac{(p_1-1)n_r}{p_1}(n_r)\right\}\subset A,
\end{equation}
into one arithmetic progression $d(\frac{n_r}{p_1})$ in $B$. Hence
$A$ is a primely split
of the ECS $B$, and this completes the proof. \qed \\
It is important to notice that there exist ECS which are not prime splitting of any ECS.
By theorem A the following example which is a particular case of \cite[Example 2.5]{LL00} is minimal.
\begin{example}
The following ECS is not a prime splitting of any ECS.
\begin{align}
A=\{2(6),4(6),1(10),3(10),7(10),9(10),0(15),\nonumber\\
5(30),6(30),12(30),18(30),24(30),25(30)\}
\end{align}
Notice that the maximal modulus is $30$. The reason that $A$ is not
a primely split of any ECS follows from the fact that there is
no way to split the following vanishing sum
\begin{equation}
\xi ^5+\xi ^6+\xi ^{12}+\xi ^{18}+\xi ^{24}+\xi ^{25}=0,
\end{equation}
where $\xi$ is a $30$-\textit{th} primitive root of unity, into two
vanishing sums (see \cite{LL00}).
\end{example}
\noindent \textit{Proof of Corollary~\ref{th:ECS3}}.\\
First, notice that for two ECS, $A$ and $B$, if  $A\models B$, then any modulus of $B$ is a divisor of a modulus of $A$.
Now, since all moduli of $A$ admit at most two prime factors, then any maximal modulus of $A$ also admits at most two prime factors.
By applying Theorem A to each maximal modulus we proceed by induction noticing that in each step of the induction
all the moduli (and hence all the maximal moduli) admit at most two prime factors until we get the trivial ECS.\qed \\
\begin{remark}
Note that Corollary~\ref{th:ECS3} gives a way to construct all ECS
with $N=p_1^{s_1}p_2^{s_2}$ for two given primes $p_1,p_2$.
\end{remark}
For the next proof note that by the Chinese remainder theorem, an
ECS cannot contain coprime moduli.\\
\textit{Proof of Corollary~\ref{th:ECS4}}.\\
 Assume that there is no modulus
 $n_j=p_1^{d_1}p_2^{d_2}p_3^{d_3}, d_1,d_2,d_3>0$.\\
Let $A=A_1$. By the hypothesis of the corollary and by the above
assumption there is a maximal modulus $p_1^{l_1}p_2^{l_2}
(l_1,l_2\geq 1)$. Hence by Theorem A there is an ECS
$A_1\models A_2$. We proceed by induction. As long as there is a
modulus $p_1^{l_1}p_2^{l_2} (l_1,l_2\geq 1)$ there is a maximal
modulus of the same form.  The sequence $A_1\models A_2\ldots
\models A_l$, must terminate. The terminal ECS, $A_l$ has no modulus
of the form $p_1^{l_1}p_2^{l_2} (l_1,l_2\geq 1)$. Hence, there is
either a modulus $p_1^{l_1} (l_1>0)$ or a modulus $p_2^{l_2}
(l_2>0)$. Since we assumed that $A$ contains moduli
$p_1^{m_3}p_3^{m_4}$ and $p_2^{m_5}p_3^{m_6}$, then, in both cases
$A_l$ contain coprime moduli. This cannot happen by the Chinese
remainder theorem.\qed

\end{document}